\documentclass[12pt,oneside,english]{amsart}
\usepackage[T1]{fontenc}
\usepackage[latin9]{inputenc}
\usepackage{mathrsfs}
\usepackage{amstext}
\usepackage{amsthm}
\usepackage{amssymb}

\makeatletter

\providecommand{\tabularnewline}{\\}

\numberwithin{equation}{section}
\numberwithin{figure}{section}
\theoremstyle{plain}
\newtheorem{thm}{\protect\theoremname}[section]
  \theoremstyle{plain}
  \newtheorem{conjecture}[thm]{\protect\conjecturename}
  \theoremstyle{plain}
  \newtheorem{cor}[thm]{\protect\corollaryname}
  \theoremstyle{definition}
  \newtheorem{example}[thm]{\protect\examplename}
  \theoremstyle{plain}
  \newtheorem{lem}[thm]{\protect\lemmaname}
  \theoremstyle{plain}
  \newtheorem{prop}[thm]{\protect\propositionname}

\usepackage{lmodern}
\usepackage{colortbl}
\definecolor{lightgray}{gray}{0.95}
\definecolor{sortagray}{gray}{0.85}
\definecolor{darkergray}{gray}{0.75}

\makeatother

\usepackage{babel}
  \providecommand{\conjecturename}{Conjecture}
  \providecommand{\corollaryname}{Corollary}
  \providecommand{\examplename}{Example}
  \providecommand{\lemmaname}{Lemma}
  \providecommand{\propositionname}{Proposition}
\providecommand{\theoremname}{Theorem}

\begin{document}

\title{On the linear span of lattice points in a parallelepiped}
\begin{abstract}
\global\long\def\ZZ{\mathbf{Z}}
\global\long\def\NN{\mathbf{N}}
\global\long\def\RR{\mathbf{R}}
\global\long\def\CC{\mathbf{C}}
\global\long\def\Hom{\mathrm{Hom}}
\global\long\def\supp{\mathrm{supp}}
\global\long\def\gauss{\tau}
\global\long\def\sgn{\mathrm{sgn}}
\global\long\def\projmapspart{\mathscr{I}}
\global\long\def\one{\mathbf{1}}
\global\long\def\spann{\mathrm{span}}
\global\long\def\zeroprofilepart{\mathscr{L}}
\global\long\def\Ltwoodd{L_{\mathrm{odd}}^{2}}
\global\long\def\divisor{\mathrm{d}}
Let $\Lambda\subset\mathbf{R}^{n}$ be a lattice which contains the
integer lattice $\mathbf{Z}^{n}$. We characterize the space of linear
functions $\mathbf{R}^{n}\rightarrow\mathbf{R}$ which vanish on the
lattice points of $\Lambda$ lying in the half-open unit cube $[0,1)^{n}$.
We also find an explicit formula for the dimension of the linear span
of $\Lambda\cap[0,1)^{n}$. The results in this paper generalize and
are based on the Terminal Lemma of Reid, which is in turn based upon
earlier work of Morrison and Stevens on the classification of four
dimensional isolated Gorenstein terminal cyclic quotient singularities.
\end{abstract}

\author{Marcel Celaya}

\address{Department of Mathematics, Georgia Institute of Technology}

\email{mcelaya@gatech.edu}


\thanks{Research supported in part by the Natural Sciences and Engineering
Research Council of Canada}
\maketitle

\section{Introduction}

Let $n$ be a positive integer and let $\Lambda$ denote a lattice
in $\RR^{n}$ that contains the integer lattice $\ZZ^{n}$. We are
interested in understanding the combinatorics of the lattice points
of $\Lambda$ inside the half-open cube 
\[
[0,1)^{n}:=\left\{ \left(x_{1},\ldots,x_{n}\right)\in\RR^{n}:\;0\leq x_{i}<1\mbox{ for all \ensuremath{i=1,2,\ldots,n}}\right\} .
\]
In general, questions about these points are difficult. For instance,
if $\Lambda=\tfrac{1}{2}\ZZ^{n}$ and $u=\left(u_{1},\ldots,u_{n}\right)\in\RR^{n}$
has integral coordinates, then the problem of deciding if there exists
a nonzero point in $\Lambda\cap[0,1)^{n}$ on the hyperplane $\left\{ x\in\RR^{n}:\left\langle u,x\right\rangle =0\right\} $
is NP-complete. Indeed, it is straightforward to reduce \textsc{subset-sum}
to this problem; such a point exists if and only if some integers
in the multiset $\left\{ u_{1},\ldots,u_{n}\right\} $ sum to zero.
As pointed out by Seb\H{o} in \cite[p. 401]{sebho_introduction_1999},
the well-known \emph{Lonely Runner Conjecture \cite{bienia_flows_1998}}
can be stated as a problem about the existence of a lattice point
in $\Lambda\cap[0,1)^{n}$ satisying certain linear inequalities where
the lattice $\Lambda$ is generated by $\ZZ^{n}$ plus a rational
vector $v\in\RR^{n}$ encoding the speeds of the runners.

Our approach to understanding the lattice points in $\Lambda\cap[0,1)^{n}$
begins with a result that is commonly attributed to G. K. White \cite{white_lattice_1964}
but was discovered independently by several others \cite{reeve_volume_1957,reznick_lattice_1986}.
It says that a tetrahedron $T$ in $\RR^{3}$ which has integral vertices
but no other integral points must be ``sandwiched'' between two
parallel lattice hyperplanes. More precisely, there exists an integral
normal vector $u=\left(u_{1},u_{2},u_{3}\right)$ and an integer $\delta$
such that two of the vertices of $T$ lie on the plane $\left\langle u,x\right\rangle =\delta$
and the other two lie on the plane $\left\langle u,x\right\rangle =\delta+1$.
We may assume that one of the vertices of the tetrahedron lies at
the origin, so that there are three integral vectors $v,v',v''$ corresponding
to the three edges of the tetrahedron incident to the origin. The
tricky part of White's theorem is to show that, after applying a unimodular
transformation (i.e. a linear transformation of $\RR^{n}$ which fixes
$\ZZ^{n}$), we may further assume that $v=(1,0,0)$, $v'=(0,1,0)$,
and $v''=(1,a,r)$ where $a$ and $r$ are positive integers such
that $a<r$ and $a$ is coprime to $r$. From there, the normal vector
$u=(1,0,0)$ establishes the conclusion of the theorem.

The triples $\lambda=(\lambda_{1},\lambda_{2},\lambda_{3})\in\RR^{3}$
such that $\lambda_{1}v+\lambda_{2}v'+\lambda_{3}v''\in\ZZ^{3}$ form
a lattice $\Lambda\subset\RR^{3}$ which contains the integer lattice
$\ZZ^{3}$. Moreover, $T$ contains a non-vertex integral point if
and only if there exists some nonzero $\lambda\in\Lambda\cap[0,1)^{3}$
such that $\lambda_{1}+\lambda_{2}+\lambda_{3}\leq1$. Indeed, such
a $\lambda$ corresponds to a proper convex combination of at least
two vertices of $T$. A short exercise shows that there are exactly
$r$ lattice points in $\Lambda\cap[0,1)^{3}$ and they are of the
form $(\{ka/r\},\{-ka/r\},\{k/r\})$ for $k=0,1,2,\ldots,d-1$. Here
$\{x\}$ denotes the fractional part of the real number $x$, the
unique real number in $[0,1)$ congruent to $x$ $\mathrm{mod}\;1$.
In particular, we can view the emptiness of $T$ as a consequence
of the fact that the first two components of every nonzero $\lambda\in\Lambda\cap[0,1)^{3}$
sum to 1 and therefore the sum $\lambda_{1}+\lambda_{2}+\lambda_{3}$
exceeds 1.

More generally, if $\Lambda\subset\RR^{n}$ is a lattice that contains
$\ZZ^{n}$, then we can think of the presence of such complementary
pairs of coordinates as a restriction on the extent to which the nonzero
points in $\Lambda\cap[0,1)^{n}$ can deviate from the hyperplane
$x_{1}+\cdots+x_{n}=n/2$. Seb\H{o} asks in \cite{sebho_introduction_1999}
about the most restrictive case, where \emph{all} the nonzero lattice
points in $\Lambda\cap[0,1)^{n}$ lie on this hyperplane. He conjectures
that this can only happen if the coordinates can be grouped into $n/2$
pairs of complementary coordinates as above. More precisely, suppose
$\Lambda$ is a lattice in $\RR^{n}$ generated by $\ZZ^{n}$ and
the point $\tfrac{1}{r}\left(a_{1},\ldots,a_{n}\right)$, where the
$a_{i}$'s are positive integers coprime to a positive integer $r$.
Note that for every $\left(\lambda_{1},\ldots,\lambda_{n}\right)\in\Lambda\cap[0,1)^{n}$,
there exists an integer $0\leq k<r$ such that $\lambda_{i}=\{ka_{i}/r\}$
for each $i=1,2,\ldots,n$. Seb\H{o} asks if the following statement
is true:
\begin{conjecture}
\label{conj:sebo_conjecture}The equality
\[
\lambda_{1}+\lambda_{2}+\cdots+\lambda_{n}=n/2
\]
holds for all nonzero $\lambda\in\Lambda\cap[0,1)^{n}$ if and only
if $n$ is even and (after possibly reordering) $a_{i}+a_{i+1}=r$
for $i=1,3,5,\ldots,n-1$. 
\end{conjecture}
In \cite{sebho_introduction_1999}, Seb\H{o} proves the case $n=4$
of his conjecture and uses it to deduce White's theorem.

It turns out, however, that Seb\H{o}'s conjecture had already been
established some years earlier by Morrison and Stevens in their paper
\cite{morrison_terminal_1984} (see also \cite{batyrev_generalization_2010}).
Although they were also primarily interested in the case $n=4$, their
proof stands out as it easily extends to all positive even integers
$n$. In \cite{morrison_terminal_1984}, Morrison and Stevens use
this result to derive a complete classification of three dimensional
isolated terminal cyclic quotient singularities and four dimensional
isolated Gorenstein terminal cyclic quotient singularities. The survey
paper of Borisov \cite{borisov_quotient_2008} provides a nice description
and some interesting number-theoretic applications of this problem. 

In \cite[Theorem 5.4]{reid_young_1985}, Reid proves a stronger version
of Conjecture \ref{conj:sebo_conjecture} that does not require the
$a_{i}$'s to be coprime to $r$. Given a lattice $\Lambda\subset\RR^{n}$
generated by $\ZZ^{n}$ and a point $\tfrac{1}{r}(a_{1},\ldots,a_{n})$
where the $a_{i}$'s are positive integers less then $r$, he finds
a characterization for when all lattice points in $\Lambda\cap[0,1)^{n}$
lie in a given hyperplane through the origin. We show in Section \ref{sec:Main-theorem}
how to deduce Seb\H{o}'s conjecture from Reid's result, known as
the \emph{Terminal Lemma}. In \cite[Section 6]{reid_young_1985},
Reid shows how the Terminal Lemma can be systematically applied to
obtain Mori's classification results on three dimensional terminal
singularities found in \cite{mori_$3$-dimensional_1985}.

Other variations of Conjecture \ref{conj:sebo_conjecture} have found
application in Ehrhart theory; in particular the problem of classifying
lattice polytopes with a given $h^{*}$-polynomial. In \cite{batyrev_lattice_2013},
Batyrev and Hofscheier give a classification of all lattice polytopes
whose $h^{*}$-polynomial is of the form $h^{*}\left(t\right)=1+ct^{k}$
for some positive integers $c,k$ in terms of particular linear codes.
This work was further developed by Higashitani, Nill, and Tsuchiya
in \cite{higashitani_gorenstein_2015} in order to obtain a combinatorial
description of Gorenstein polytopes with a trinomial $h^{*}$-polynomial.
A key ingredient of these results is a version of Conjecture \ref{conj:sebo_conjecture}
applicable to lattices $\Lambda\subset\RR^{n}$ containing $\ZZ^{n}$
with the property that the quotient group $\Lambda/\ZZ^{n}$ is isomorphic
to the additive group of a finite field.

In this paper, we establish of a variation of Conjecture \ref{conj:sebo_conjecture}
which imposes no restrictions on the lattice $\Lambda\subset\RR^{n}$
except that it must contain the integer lattice $\ZZ^{n}$. The paper
is organized into six sections. Following the introduction, Section
\ref{sec:Background-and-notation} outlines the basic notation and
concepts we use. In Section \ref{sec:Main-theorem}, we state our
main theorem, which directly generalizes Reid's Terminal Lemma by
dropping the assumption $\Lambda/\ZZ^{n}$ must be cyclic. From our
theorem we deduce a formula for the dimension of the linear span of
the points in $\Lambda\cap[0,1)^{n}$. We also state a natural generalization
of Conjecture \ref{conj:sebo_conjecture} when there are no assumptions
on the group structure of $\Lambda/\ZZ^{n}$. Finally, we state the
two main technical tools needed to prove our main theorem. Section
\ref{sec:Proof-of-Theorem-Main} outlines the proof of our main theorem
using these two tools, both of which are statements about an arbitrary
additive finite abelian group $G$. Section \ref{sec:Proof-of-First-Technical-Tool}
contains the proof of the first technical tool, which asserts that
a specific collection of indicator functions defined on $G$ is linearly
independent. In Section \ref{sec:Proof-of-Second-Technical-Tool}
is the proof of the second technical tool, which gives a specific
spanning set for the space of functions $f:G\rightarrow\CC$ satisfying
$f(-a)=-f(a)$ for all $a\in G$. At a high level, we mostly follow
the path laid out by Reid in \cite{reid_young_1985}. We differ somewhat
in the details, however, by making liberal use of the results in \cite[Section 9.2]{montgomery_multiplicative_2006}.

\section{\label{sec:Background-and-notation}Background and notation}

\subsection{Notation}

For $u=(u_{1},\ldots,u_{n})\in\RR^{n}$ and $v=(v_{1},\ldots,v_{n})\in\RR^{n}$,
we let $\left\langle u,v\right\rangle =u_{1}v_{1}+\cdots+u_{n}v_{n}$
denote the usual dot product. If $v$ is a vector in a vector space
with specified coordinates, then $\supp\left(v\right)$ denotes the
set of coordinates $i$ for which $v_{i}\neq0$. For $x\in\RR$, we
define $\left\{ x\right\} $ to be be the unique real number in the
half-open interval $[0,1)$ in which $x-\left\{ x\right\} $ is an
integer. We frequently make use of the fact that for any $x\in\RR$,
$\{x\}+\{1-x\}$ equals $1$ if $x\notin\ZZ$ and 0 otherwise. We
define the first periodic Bernoulli function $B_{1}:\RR\rightarrow\RR$
by
\[
B_{1}\left(x\right):=\begin{cases}
\left\{ x\right\} -1/2, & x\notin\ZZ\\
0, & x\in\ZZ.
\end{cases}
\]
For $x+\ZZ\in\RR/\ZZ$, we also define $B_{1}\left(x+\ZZ\right):=B_{1}\left(x\right)$. 

For a finite group $G$, we denote the space of complex functions
$f:G\rightarrow\CC$ by $L^{2}(G)$ which forms a vector space under
pointwise addition and comes with the inner product
\[
\left\langle f,h\right\rangle =\frac{1}{\left|G\right|}\sum_{g\in G}f(g)\overline{h(g)}.
\]

\subsection{Character theory of finite abelian groups}

We refer the reader to \cite{conrad2010characters} and \cite{babai_fourier_1989}
for an introduction to the character theory of finite abelian groups,
and record some key facts here. For a finite abelian group $G$, let
$\widehat{G}$ denote the multiplicative group of homomorphisms $G\rightarrow\CC^{\times}$
from $G$ to the nonzero complex numbers. The group operation of $\widehat{G}$
is given by pointwise multiplication: $\left(\chi\psi\right)\left(g\right):=\chi\left(g\right)\psi\left(g\right)$
for each $g\in G$ and for each $\chi,\psi\in\widehat{G}$. The inverse
of $\chi\in\widehat{G}$ satisfies $\chi^{-1}(g)=\overline{\chi(g)}$
for all $g\in G$; we therefore denote $\chi^{-1}$ by $\overline{\chi}$.
Elements in $\widehat{G}$ are called \emph{characters} of $G$. There
is an isomorphism $G\simeq\widehat{G}$ and we identify $G$ with
$\widehat{\widehat{\mbox{\textit{G}}}}$ via the natural isomorphism
which takes $g\in G$ to the point evaluation map $\left(\chi\mapsto\chi(g)\right)\in\widehat{\widehat{\mbox{\textit{G}}}}$.
For a subgroup $K$ of $G$, let 
\[
K^{\perp}:=\{\chi\in\widehat{G}:\chi\left(k\right)=1\;\mbox{for all \ensuremath{k\in K}}\}
\]
which is a subgroup of $\widehat{G}$. With the above identification
of $G$ and $\widehat{\widehat{\mbox{\textit{G}}}}$, we have $K^{\perp\perp}:=(K^{\perp})^{\perp}=K$. 

We define $e:\RR/\ZZ\rightarrow\CC^{\times}$ to be the injective
group homomorphism $x+\ZZ\mapsto\exp\left(2\pi ix\right)$. Since
the additive group $\RR/\ZZ$ embeds into the multiplicative group
$\CC^{\times}$ of nonzero complex numbers via the map $x+\ZZ\mapsto e\left(x\right)$,
it follows that the additive group $H:=\Hom_{\ZZ}\left(G,\RR/\ZZ\right)$
is isomorphic to the multiplicative group $\widehat{G}$ via the map
$\phi\mapsto e\circ\phi$. In this paper it will typically be more
convenient to state results and proofs in terms of $H$ rather than
$\widehat{G}$. However, we will sometimes take advantage of both
the multiplicative and additive structure offered by $\CC$ and work
with $\widehat{G}$ instead.

\section{\label{sec:Main-theorem}Overview of results}

We begin with our generalization of Reid's Terminal Lemma as claimed
in the abstract. Let $\Lambda\subset\RR^{n}$ be a lattice containing
$\ZZ^{n}$. For $i=1,2,\ldots,n$, let $\pi_{i}:\Lambda/\ZZ^{n}\rightarrow\RR/\ZZ$
denote the coordinate projection map sending $\left(\lambda_{1},\ldots,\lambda_{n}\right)+\ZZ^{n}$
to $\lambda_{i}+\ZZ$. Observe that these maps are homomorphisms in
the additive group $\Hom_{\ZZ}\left(\Lambda/\ZZ^{n},\RR/\ZZ\right)$
under pointwise addition; thus, it makes sense to talk about $-\pi_{i}$
for each $i$. By restricting to the appropriate subspace of $\RR^{n}$,
we assume without loss of generality that $\ker\pi_{i}\neq\Lambda/\ZZ^{n}$
for any $i$.
\begin{thm}[{Terminal Lemma, c.f. \cite[Theorem 5.4]{reid_young_1985}}]
\label{thm:main-thm}Let $u=(u_{1},\ldots,u_{n})\in\RR^{n}$. Then
$\left\langle u,\lambda\right\rangle =0$ for every $\lambda\in\Lambda\cap[0,1)^{n}$
if and only if
\begin{equation}
\sum_{\substack{i=1\\
\pi_{i}=\pi_{j}
}
}^{n}u_{i}=\sum_{\substack{i=1\\
\pi_{i}=-\pi_{j}
}
}^{n}u_{i}\label{eq:thm1_group_into_pairs}
\end{equation}
and
\begin{equation}
\sum_{\substack{i=1\\
\ker\pi_{i}=\ker\pi_{j}
}
}^{n}u_{i}=0\label{eq:thm1_kernel_sum_to_zero}
\end{equation}
for each $j=1,2,\ldots,n$.
\end{thm}
From this theorem several corollaries can be deduced. The first shows
that the dimension of the span of the lattice points of $\Lambda$
in the half-open unit cube $[0,1)^{n}$ can be computed explicitly
in terms of the coordinate projection functions $\pi_{i}:\Lambda/\ZZ^{n}\rightarrow\RR/\ZZ$.

Let $\mathscr{I}$ denote the equivalence classes of the equivalence
relation on the coordinates $\{1,2,\ldots,n\}$ in which $i\sim j$
in $\mathscr{I}$ if and only if $\pi_{i}=\pi_{j}$ or $\pi_{i}=-\pi_{j}$.
Let $\mathscr{K}$ denote the equivalence classes of the equivalence
relation on the coordinates $\{1,2,\ldots,n\}$ where $i\sim j$ in
$\mathscr{K}$ if and only if $\ker(\pi_{i})=\ker(\pi_{j})$. Note
that $\mathscr{K}$ coarsens $\mathscr{I}$ since $\ker(\pi_{i})=\ker(-\pi_{i})$
for all $i$.
\begin{cor}
\label{cor:The-dimension-of-span}The dimension of $\spann(\Lambda\cap[0,1)^{n})$
is equal to $\iota+\kappa$, where $\iota$ denotes the number of
equivalence classes $[i]\in\mathscr{I}$ such that $\pi_{i}\neq-\pi_{i}$
and $\kappa$ denotes the number of equivalence classes in $\mathscr{K}$
containing coordinates $i,j$ (possibly equal) in which $\pi_{i}=-\pi_{j}$.
\end{cor}
\begin{proof}
The distinct relations of the form (\ref{eq:thm1_group_into_pairs})
are in 1-1 correspondence with the equivalence classes $[i]\in\mathscr{I}$
such that $\pi_{i}\neq-\pi_{i}$. Note that in case $\pi_{i}=-\pi_{i}$,
the relation (\ref{eq:thm1_group_into_pairs}) is trivial. Similarly,
the distinct relations of the form (\ref{eq:thm1_kernel_sum_to_zero})
are in 1-1 correspondence with the equivalence classes of $\mathscr{K}$.
The collection of all these relations are linearly independent except
in the case when some $J\in\mathscr{K}$ does not contain any coordinates
$i,j$ for which $\pi_{i}=-\pi_{j}$. In this situation, the relation
(\ref{eq:thm1_kernel_sum_to_zero}) corresponding to $J$ is already
implied by the relations (\ref{eq:thm1_group_into_pairs}) corresponding
to the equivalence classes $I\in\mathscr{I}$ contained in $J$. Thus,
after excluding the relations corresponding to such $J\in\mathscr{K}$,
we conclude that the space of $u\in\RR^{n}$ in which $\left\langle u,\lambda\right\rangle =0$
for all $\lambda\in\Lambda\cap[0,1)^{n}$ has dimension $n-\iota-\kappa$
and hence the dimension of the span of $\Lambda\cap[0,1)^{n}$ equals
$\iota+\kappa$. 
\end{proof}
\begin{table}
\centering{}\caption{An illustration of Theorem \ref{thm:main-thm} and Corollary \ref{cor:The-dimension-of-span}
for the lattice $\Lambda$ generated by $\protect\ZZ^{8}$ and the
two points $\lambda=\tfrac{1}{10}(1,9,3,7,1,1,3,5)$ and $\lambda'=\tfrac{1}{10}(2,8,6,4,1,1,3,0)$.
The coordinate projection maps $\pi_{i}:\Lambda/\protect\ZZ^{8}\rightarrow\protect\RR/\protect\ZZ$
are uniquely determined by the two numbers $\lambda_{i}$ and $\lambda'_{i}$.
In this example, $\iota=4$ (corresponding to the classes $I_{1},I_{2},I_{3},I_{4}$
in $\mathscr{I}$) and $\kappa=2$ (corresponding to the classes $J_{1},J_{3}$
in $\mathscr{K}$). By Corollary \ref{cor:The-dimension-of-span},
the dimension of the linear span of $\Lambda\cap[0,1)^{8}$ is 6.}
\begin{tabular}{|c||ccccc|cccc|cc|}
\hline 
{\scriptsize{}$\mathscr{K}$} & \multicolumn{4}{c}{{\scriptsize{}\cellcolor{darkergray}$J_{1}$}} &  & \multicolumn{3}{c}{{\scriptsize{}\cellcolor{darkergray}$J_{2}$}} &  & {\scriptsize{}\cellcolor{darkergray}$J_{3}$} & \tabularnewline
{\scriptsize{}$\mathscr{I}$} & \multicolumn{2}{c}{{\scriptsize{}\cellcolor{lightgray}$I_{1}$}} & \multicolumn{2}{c}{{\scriptsize{}\cellcolor{sortagray}$I_{2}$}} &  & \multicolumn{2}{c}{{\scriptsize{}\cellcolor{lightgray}$I_{3}$}} & {\scriptsize{}\cellcolor{sortagray}$I_{4}$} &  & {\scriptsize{}\cellcolor{lightgray}$I_{5}$} & \tabularnewline
{\scriptsize{}$i$} & {\scriptsize{}\cellcolor{lightgray}$1$} & {\scriptsize{}\cellcolor{lightgray}$2$} & {\scriptsize{}\cellcolor{sortagray}$3$} & {\scriptsize{}\cellcolor{sortagray}$4$} &  & {\scriptsize{}\cellcolor{lightgray}$5$} & {\scriptsize{}\cellcolor{lightgray}$6$} & {\scriptsize{}\cellcolor{sortagray}$7$} &  & {\scriptsize{}\cellcolor{lightgray}$8$} & \tabularnewline
\hline 
{\scriptsize{}$\lambda_{i}$} & {\scriptsize{}\cellcolor{lightgray}$0.1$} & {\scriptsize{}\cellcolor{lightgray}$0.9$} & {\scriptsize{}\cellcolor{sortagray}$0.3$} & {\scriptsize{}\cellcolor{sortagray}$0.7$} &  & {\scriptsize{}\cellcolor{lightgray}$0.1$} & {\scriptsize{}\cellcolor{lightgray}$0.1$} & {\scriptsize{}\cellcolor{sortagray}$0.3$} &  & {\scriptsize{}\cellcolor{lightgray}$0.5$} & \tabularnewline
{\scriptsize{}$\lambda'_{i}$} & {\scriptsize{}\cellcolor{lightgray}$0.2$} & {\scriptsize{}\cellcolor{lightgray}$0.8$} & {\scriptsize{}\cellcolor{sortagray}$0.6$} & {\scriptsize{}\cellcolor{sortagray}$0.4$} &  & {\scriptsize{}\cellcolor{lightgray}$0.1$} & {\scriptsize{}\cellcolor{lightgray}$0.1$} & {\scriptsize{}\cellcolor{sortagray}$0.3$} &  & {\scriptsize{}\cellcolor{lightgray}$0.0$} & \tabularnewline
\hline 
{\scriptsize{}(\ref{eq:thm1_group_into_pairs}) relations} & {\scriptsize{}\cellcolor{lightgray}$\phantom{+}u_{1}$} & {\scriptsize{}\cellcolor{lightgray}$-u_{2}$} & {\scriptsize{}\cellcolor{sortagray}} & {\scriptsize{}\cellcolor{sortagray}} & {\scriptsize{}$=0$} & {\scriptsize{}\cellcolor{lightgray}$\phantom{+}u_{5}$} & {\scriptsize{}\cellcolor{lightgray}$+u_{6}$} & {\scriptsize{}\cellcolor{sortagray}} & {\scriptsize{}$=0$} & {\scriptsize{}\cellcolor{lightgray}$u_{8}-u_{8}$} & {\scriptsize{}$=0$}\tabularnewline
 & {\scriptsize{}\cellcolor{lightgray}} & {\scriptsize{}\cellcolor{lightgray}} & {\scriptsize{}\cellcolor{sortagray}$\phantom{+}u_{3}$} & {\scriptsize{}\cellcolor{sortagray}$-u_{4}$} & {\scriptsize{}$=0$} & {\scriptsize{}\cellcolor{lightgray}} & {\scriptsize{}\cellcolor{lightgray}} & {\scriptsize{}\cellcolor{sortagray}$\phantom{+}u_{7}$} & {\scriptsize{}$=0$} & {\scriptsize{}\cellcolor{lightgray}} & \tabularnewline
\hline 
{\scriptsize{}(\ref{eq:thm1_kernel_sum_to_zero}) relations} & {\scriptsize{}\cellcolor{darkergray}$\phantom{+}u_{1}$} & {\scriptsize{}\cellcolor{darkergray}$+u_{2}$} & {\scriptsize{}\cellcolor{darkergray}$+u_{3}$} & {\scriptsize{}\cellcolor{darkergray}$+u_{4}$} & {\scriptsize{}$=0$} & {\scriptsize{}\cellcolor{darkergray}$\phantom{+}u_{5}$} & {\scriptsize{}\cellcolor{darkergray}$+u_{6}$} & {\scriptsize{}\cellcolor{darkergray}$+u_{7}$} & {\scriptsize{}$=0$} & {\scriptsize{}\cellcolor{darkergray}$u_{8}$} & {\scriptsize{}$=0$}\tabularnewline
\hline 
\end{tabular}
\end{table}
Observe that it is always true that 
\[
\lambda_{1}+\cdots+\lambda_{n}+\mu_{1}+\cdots+\mu_{n}=\left|\supp\left(\lambda\right)\right|=\left|\supp\left(\mu\right)\right|
\]
for every pair $\lambda,\mu\in\Lambda\cap[0,1)^{n}$ for which $\lambda+\ZZ^{n}=-\mu+\ZZ^{n}$.
This follows from the fact that for every $i=1,2,\ldots,n$, we have
either $\lambda_{i}=1-\mu_{i}$ if both $\lambda_{i}$ and $\mu_{i}$
are nonzero, or $\lambda_{i}=\mu_{i}=0$ otherwise. The next corollary
characterizes the situation where the ``mass'' of $\lambda+\mu$
is distributed as equally as possible between $\lambda$ and $\mu$
for all such pairs $\lambda,\mu$. It is a direct generalization of
Seb\H{o}'s Conjecture \ref{conj:sebo_conjecture}:
\begin{cor}[{c.f. \cite[Conjecture 4.1]{sebho_introduction_1999}, \cite[Proposition 1.8]{batyrev_generalization_2010}}]
\label{thm:semistable}The equality
\begin{equation}
\lambda_{1}+\cdots+\lambda_{n}=\frac{\left|\supp\left(\lambda\right)\right|}{2}\label{eq:sebo_equation}
\end{equation}
holds for all $\lambda=\left(\lambda_{1},\ldots,\lambda_{n}\right)\in\Lambda\cap[0,1)^{n}$
if and only if there exists an involution $\sigma$ of $\left\{ 1,2,\ldots,n\right\} $
(i.e. a bijection satisfying $\sigma=\sigma^{-1}$) such that $\lambda_{i}+\lambda_{\sigma\left(i\right)}$
is an integer for all $i=1,2,\ldots,n$ and $\lambda\in\Lambda$.
\end{cor}
\begin{example}
If the coordinates of the points in $\Lambda$ are all half-integral
(i.e. $\Lambda\subset\tfrac{1}{2}\ZZ^{n}$), then Corollary \ref{thm:semistable}
is trivial. Indeed, both the hypothesis and the conclusion always
hold; for the conclusion we may take $\sigma$ to be the identity
map.
\end{example}
\begin{example}
If $\Lambda$ is generated by $\ZZ^{n}$ and the point $\tfrac{1}{r}(a_{1},a_{2},\ldots,a_{n})$,
where the $a_{i}$'s are positive integers coprime to $r$, then we
recover Conjecture \ref{conj:sebo_conjecture}. Indeed, in this case
$\left|\supp(\lambda)\right|=n$ for every nonzero $\lambda\in[0,1)^{n}\cap\Lambda$,
and $\lambda_{i}+\lambda_{\sigma(i)}\in\ZZ$ for all $i=1,2,\ldots,n$
and $\lambda\in\Lambda$ if and only if $a_{i}+a_{\sigma(i)}=r$ for
all $i=1,2,\ldots,n$. This follows from the fact that for every nonzero
$\lambda\in\Lambda\cap[0,1)^{n}$ there exists $1\leq k\leq r-1$
such that $\lambda_{i}=\{ka_{i}/r\}\neq0$ for each $i=1,2,\ldots,n$.
\end{example}
\begin{example}
If $\Delta\subseteq\RR^{n-1}$ is a lattice polytope, then the Ehrhart
series of $\Delta$ is given by
\[
\mathrm{Ehr}_{\Delta}(t)=\sum_{m\geq0}\left|m\Delta\cap\ZZ^{n-1}\right|t^{m}=\frac{1+h_{1}^{*}t+\cdots+h_{n-1}^{*}t^{n-1}}{(1-t)^{n}}
\]
and the numerator of the left-hand side is called the $h^{*}$\emph{-polynomial}
of $\Delta$. If furthermore $\Delta=\mathrm{conv}(v_{1},\ldots,v_{n})$
is a simplex, where each $v_{i}\in\ZZ^{n-1}$, then it is known that
$h_{k}^{*}$ equals the number of $\lambda\in\Lambda_{\Delta}\cap[0,1)^{n}$
such that $\lambda_{1}+\cdots+\lambda_{n}=k$ \cite[Corollary 3.11]{beck_computing_2007}.
Here $\Lambda_{\Delta}\subset\RR^{n}$ denotes the dual lattice of
the lattice generated by $(v_{i},1)\in\ZZ^{n}$ for $i=1,2,\ldots,n$;
equivalently, the lattice of points $\lambda\in\RR^{n}$ such that
$\left\langle \lambda,(v_{i},1)\right\rangle \in\ZZ$ for all $i$.

Polytopes $\Delta$ with $h^{*}$-polynomial of the form $1+h_{k}^{*}t^{k}$
for some positive $k$ have been completely classified by Batyrev
and Hofscheier \cite{batyrev_lattice_2013}. They show that such a
polytope $\Delta$ must be a simplex; therefore, the corresponding
lattice $\Lambda_{\Delta}$ has the property that $\lambda_{1}+\cdots+\lambda_{n}=k$
for all nonzero $\lambda\in\Lambda_{\Delta}\cap[0,1)^{n}$. It follows
that the hypothesis of Corollary \ref{thm:semistable} applies to
$\Lambda_{\Delta}$, and the resulting involution $\sigma$ appears
in their classification. They also describe some properties that $\Lambda_{\Delta}/\ZZ^{n}$
has; for instance, $\Lambda_{\Delta}/\ZZ^{n}$ is isomorphic to the
additive group of $\mathbf{F}_{p}^{r}$ for some prime $p$ and integer
$r$, and the integer $k$ satisfies $(p^{r}-p^{r-1})n=2k(p^{r}-1)$.
\end{example}
\begin{proof}[Proof of Corollary \ref{thm:semistable}]
The ``if'' direction is an immediate consequence of the fact that,
for every $x\in\RR$, $\{x\}+\{-x\}$ equals 1 if $x\notin\ZZ$ and
0 otherwise. 

For the ``only if'' direction, consider the lattice $\Lambda'\subset\RR^{2n}$
which is generated by $\ZZ^{2n}$ and the image of the map $\Lambda\rightarrow\RR^{2n}$
defined by
\[
(\lambda_{1},\ldots,\lambda_{n})\mapsto(\lambda_{1},\ldots,\lambda_{n},-\lambda_{1},\ldots,-\lambda_{n}).
\]
Let $\lambda\in\Lambda\cap[0,1)^{n}$ and let $\lambda'\in\Lambda'\cap[0,1)^{2n}$
be the unique integral translate in $[0,1)^{2n}$ of the image of
$\lambda$ under this map. Let $\mu\in\Lambda\cap[0,1)^{n}$ the the
unique lattice point in $[0,1)^{n}$ which satisfies $\lambda+\ZZ^{n}=-\mu+\ZZ^{n}$.
Then
\[
\lambda_{1}'+\cdots+\lambda_{n}'=\lambda_{n+1}'+\cdots+\lambda_{2n}'
\]
since by assumption we have
\[
\lambda_{1}'+\cdots+\lambda_{n}'=\lambda_{1}+\cdots+\lambda_{n}=\frac{\left|\supp(\lambda)\right|}{2}
\]
and
\[
\lambda_{n+1}'+\cdots+\lambda_{2n}'=\mu_{1}+\cdots+\mu_{n}=\frac{\left|\supp(\mu)\right|}{2}
\]
and we know by the preceding discussion that $\left|\supp(\mu)\right|=\left|\supp(\lambda)\right|$.
If we let
\[
u'=(\underset{n}{\underbrace{1,\ldots,1}},\underset{n}{\underbrace{-1,\ldots,-1}})\in\RR^{2n},
\]
we get $\left\langle u',\lambda'\right\rangle =0$ for each $\lambda'\in\Lambda'\cap[0,1)^{2n}$.
We may therefore apply Theorem \ref{thm:main-thm} to obtain the equality
\[
\left(\sum_{\substack{i=1\\
\pi_{i}=\pi_{j}
}
}^{n}1\right)-\left(\sum_{\substack{i=1\\
\pi_{i}=-\pi_{j}
}
}^{n}1\right)=\left(\sum_{\substack{i=1\\
\pi_{i}=-\pi_{j}
}
}^{n}1\right)-\left(\sum_{\substack{i=1\\
\pi_{i}=\pi_{j}
}
}^{n}1\right)
\]
for each $j\in\{1,2,\ldots,n\}$, which simplifies to
\[
\left|\left\{ i:\pi_{i}=\pi_{j}\right\} \right|=\left|\left\{ i:\pi_{i}=-\pi_{j}\right\} \right|.
\]
We now construct our involution $\sigma:\left\{ 1,2,\ldots,n\right\} \rightarrow\left\{ 1,2,\ldots,n\right\} $.
For each $i$ such that $\pi_{i}=-\pi_{i}$, we set $\sigma(i)=i$.
For each coordinate projection map $\pi$ such that $\pi\neq-\pi$,
we pair up each coordinate $i$ such that $\pi_{i}=\pi$ with a unique
coordinate $j$ such that $\pi_{j}=-\pi$. Then, for each such pair
$(i,j)$, we set $\sigma(i)=j$ and $\sigma(j)=i$. Now let $\lambda\in\Lambda$
and let $i\in\{1,2,\ldots,n\}$. Then
\[
\lambda_{i}+\lambda_{\sigma(i)}+\ZZ=\pi_{i}(\lambda)+(-\pi_{i})(\lambda)=0+\ZZ,
\]
and hence $\lambda_{i}+\lambda_{\sigma(i)}$ is an integer.
\end{proof}
The proof of Theorem \ref{thm:main-thm}, specifically the ``only-if''
direction, depends on the following two claims. The first is used
to establish the relations (\ref{eq:thm1_kernel_sum_to_zero}) assuming
the hypotheses of Theorem \ref{thm:main-thm}. The proof given in
the next section relies on the Poisson summation formula for finite
abelian groups.
\begin{lem}
\label{prop:01matrixinvertible}Let $G$ be a finite abelian group.
For a subgroup $K$ of $G$, let $\one_{K}\in L^{2}(G)$ denote the
indicator function of $K$. Then
\[
\{\one_{K}\;:\;K^{\perp}\mathrm{\;is\;a\;cyclic\;subgroup\;of\;}\widehat{G}\}
\]
is linearly independent in $L^{2}(G).$
\end{lem}
We remark that this statement is quite easy to prove in the case when
$G$ is cyclic.

The second claim is used to establish the relations (\ref{eq:thm1_group_into_pairs})
assuming the hypotheses of Theorem \ref{thm:main-thm}. Suppose $G$
is a finite abelian group and let $H=\Hom_{\ZZ}(G,\RR/\ZZ)$. Consider
the space $\Ltwoodd(H)$ consisting of functions $f:H\rightarrow\CC$
which satisfy $f(-\phi)=-f(\phi)$ for all $\phi\in H$. For each
$g\in G$, define the function $S_{g}\in L^{2}(H)$ by 
\[
S_{g}(\phi)=B_{1}(\phi(g))\qquad\text{for all}\;\phi\in H.
\]
Note that these functions lie in $L_{\mathrm{odd}}^{2}(H)$ since
$B_{1}$ is an odd function. Crucially, however, much more is true:
\begin{thm}[{c.f. \cite[Proposition 1.2]{morrison_terminal_1984}}]
\label{thm:dimensionMS} The space $L_{\mathrm{odd}}^{2}(H)$ is
spanned by the functions $S_{g}$ for $g\in G$.
\end{thm}
We remark that these functions are closely related to the Stickelberger
distribution associated with $B_{1}$ described in \cite[Chapter 2]{lang_cyclotomic_1990}.
As in \cite{morrison_terminal_1984,reid_young_1985}, the proof of
this theorem relies on Dirichlet's theorem that $L(1,\chi)\neq0$
for a nontrivial Dirichlet character $\chi$ where $L(s,\chi)$ denotes
the Dirichlet $L$-function associated with $\chi$. 

\section{\label{sec:Proof-of-Theorem-Main}Proof of Theorem \ref{thm:main-thm}}

We make some preliminary observations before stating the proof. Let
$H=\Hom_{\ZZ}\left(\Lambda/\ZZ^{n},\RR/\ZZ\right)$. Given $u=(u_{1},\ldots,u_{n})\in\RR^{n}$,
define $h_{u}\in L^{2}(H)$ to be the function
\[
h_{u}(\phi)=\left(\sum_{\substack{i=1\\
\pi_{i}=\phi
}
}^{n}u_{i}\right)-\left(\sum_{\substack{i=1\\
\pi_{i}=-\phi
}
}^{n}u_{i}\right)\quad\text{for all}\;\phi\in H.
\]
Also define as above, for each $\lambda\in\Lambda\cap[0,1)^{n}$,
the function $S_{\lambda}:H\rightarrow\CC$:
\[
S_{\lambda}(\phi)=B_{1}(\phi(\lambda+\ZZ^{n}))=\begin{cases}
\{\phi(\lambda+\ZZ^{n})\}-1/2, & \phi(\lambda+\ZZ^{n})\neq0+\ZZ\\
0, & \phi(\lambda+\ZZ^{n})=0+\ZZ
\end{cases}
\]
The most important property about these functions is that they are
odd functions; we have $S_{\lambda}(-\phi)=-S_{\lambda}(\phi)$ for
each $\phi\in H$ and $\lambda\in\Lambda\cap[0,1)^{n}$. 

Observe that for any $\lambda\in\Lambda\cap[0,1)^{n}$, we have
\[
\sum_{i=1}^{n}u_{i}\lambda_{i}=\left(\sum_{i=1}^{n}u_{i}S_{\lambda}(\pi_{i})\right)+\frac{1}{2}\left(\sum_{\substack{i=1\\
\lambda_{i}\neq0
}
}^{n}u_{i}\right).
\]
Since $\phi\mapsto-\phi$ is a permutation of $H$, we may write the
first term as
\begin{align*}
\sum_{\phi\in H}\sum_{\substack{i=1\\
\pi_{i}=\phi
}
}^{n}u_{i}S_{\lambda}(\phi) & =\frac{1}{2}\sum_{\phi\in H}\left(\left(\sum_{\substack{i=1\\
\pi_{i}=\phi
}
}^{n}u_{i}S_{\lambda}(\phi)\right)+\left(\sum_{\substack{i=1\\
\pi_{i}=-\phi
}
}^{n}u_{i}S_{\lambda}(-\phi)\right)\right)\\
 & =\frac{1}{2}\sum_{\phi\in H}\left(\left(\sum_{\substack{i=1\\
\pi_{i}=\phi
}
}^{n}u_{i}\right)-\left(\sum_{\substack{i=1\\
\pi_{i}=-\phi
}
}^{n}u_{i}\right)\right)S_{\lambda}(\phi)\\
 & =\frac{\left|H\right|}{2}\left\langle h_{u},S_{\lambda}\right\rangle 
\end{align*}
where the second-to-last equality follows from the fact that $S_{\lambda}$
is an odd function. So we conclude that for any $u\in\RR^{n}$ with
corresponding $h_{u}\in L^{2}(H)$ as defined above, and for any $\lambda\in\Lambda\cap[0,1)^{n}$,
we have 
\begin{equation}
\sum_{i=1}^{n}u_{i}\lambda_{i}=\frac{1}{2}\left(\left|H\right|\left\langle h_{u},S_{\lambda}\right\rangle +\sum_{\substack{i=1\\
\lambda_{i}\neq0
}
}^{n}u_{i}\right).\label{eq:alternate_ulambda}
\end{equation}

\begin{proof}[Proof of the if direction of Theorem \ref{thm:main-thm}]
We start with the easier direction. Assume $u\in\RR^{n}$ satisfy
the relations (\ref{eq:thm1_group_into_pairs}) and (\ref{eq:thm1_kernel_sum_to_zero})
and let $\lambda\in\Lambda\cap[0,1)^{n}$. The relations (\ref{eq:thm1_group_into_pairs})
imply that $h_{u}$ is the zero function, so by (\ref{eq:alternate_ulambda})
we may therefore write
\[
\sum_{i=1}^{n}u_{i}\lambda_{i}=\frac{1}{2}\sum_{\substack{i=1\\
\lambda_{i}\neq0
}
}^{n}u_{i}=\frac{1}{2}\sum_{K}\sum_{\substack{i=1\\
\lambda_{i}\neq0\\
\ker\pi_{i}=K
}
}^{n}u_{i}=\frac{1}{2}\sum_{\substack{K\\
\lambda+\ZZ^{n}\notin K
}
}\sum_{\substack{i=1\\
\ker\pi_{i}=K
}
}^{n}u_{i}
\]
where the outer sums are over all subgroups $K\in\left\{ \ker\pi_{i}:i=1,2,\ldots,n\right\} $.
By (\ref{eq:thm1_kernel_sum_to_zero}), the inner sums of the double
sum on the right always vanish, and therefore the whole expression
equals zero.
\end{proof}
\begin{proof}[Proof of the only if direction]
Let $u\in\RR^{n}$ with corresponding $h_{u}\in L^{2}(H)$ as defined
above, and assume that $\left\langle u,\lambda\right\rangle =0$ for
every $\lambda\in\Lambda\cap[0,1)^{n}$. For every pair $\lambda,\mu\in\Lambda\cap[0,1)^{n}$
such that $\lambda+\ZZ^{n}=-\mu+\ZZ^{n}$, we have
\begin{equation}
\sum_{\substack{i=1\\
\lambda_{i}\neq0
}
}^{n}u_{i}=\left(\sum_{i=1}^{n}u_{i}\lambda_{i}\right)+\left(\sum_{i=1}^{n}u_{i}\mu_{i}\right)=0\label{eq:sum_u_i_zero}
\end{equation}
by our assumption that both the terms in the middle vanish. Hence
\[
\frac{\left|H\right|}{2}\left\langle h_{u},S_{\lambda}\right\rangle =\sum_{i=1}^{n}u_{i}\lambda_{i}=0
\]
for every $\lambda\in\Lambda\cap[0,1)^{n}$ by equation (\ref{eq:alternate_ulambda}).
So by Theorem \ref{thm:dimensionMS}, $h_{u}$ is orthogonal to every
odd function in $L^{2}(H)$ and therefore must be an even function.
But $h_{u}$ is an odd function by definition. It follows $h_{u}$
must be the zero map, and therefore the relations (\ref{eq:thm1_group_into_pairs})
hold.

We next show that the relations (\ref{eq:thm1_kernel_sum_to_zero})
hold as well. From (\ref{eq:sum_u_i_zero}), we have
\[
\left(\sum_{i=1}^{n}u_{i}\right)-\sum_{\substack{i=1\\
\lambda_{i}=0
}
}^{n}u_{i}=0
\]
for all $\lambda\in\Lambda\cap[0,1)^{n}$ which implies
\[
\left(\sum_{i=1}^{n}u_{i}\right)\one_{\Lambda/\ZZ^{n}}-\sum_{i=1}^{n}u_{i}\one_{\ker\pi_{i}}=\mathbf{0}
\]
where $\one_{K}\in L^{2}(\Lambda/\ZZ^{n})$ denotes the indicator
function of the subgroup $K$ of $\Lambda/\ZZ^{n}$ and $\mathbf{0}$
denotes the zero map. We may rewrite this sum as
\[
\left(\sum_{i=1}^{n}u_{i}\right)\one_{\Lambda/\ZZ^{n}}-\sum_{K}\left(\sum_{\substack{i=1\\
\ker\pi_{i}=K
}
}^{n}u_{i}\right)\one_{K}=\mathbf{0}
\]
where the second sum is over all subgroups $K\in\left\{ \ker\pi_{i}:i=1,2,\ldots,n\right\} $. 

Let $e:\RR/\ZZ\rightarrow\CC^{\times}$ be the map $x+\ZZ\mapsto\exp\left(2\pi ix\right)$.
Then $e\circ\pi_{i}\in\widehat{\Lambda/\ZZ^{n}}$ and, moreover, $\ker\pi_{i}=\left\langle e\circ\pi_{i}\right\rangle ^{\perp}$
for each $i=1,2,\ldots,n$. We also have $\Lambda/\ZZ^{n}=\left\langle \chi_{0}\right\rangle ^{\perp}$,
where $\chi_{0}$ denotes the identity of $\widehat{\Lambda/\ZZ^{n}}$.
It follows that $(\ker\pi_{i})^{\perp}$ for $i=1,2,\ldots,n$ and
$(\Lambda/\ZZ^{n})^{\perp}$ are all cyclic subgroups of $\widehat{\Lambda/\ZZ^{n}}$.
By Lemma \ref{prop:01matrixinvertible}, then, the set of indicator
functions in the above linear combination are linearly independent.
We conclude each of the coefficients of the indicator functions above
are zero, and therefore the relations (\ref{eq:thm1_kernel_sum_to_zero})
hold. Note that there is no $\one_{\Lambda/\ZZ^{n}}$ term among the
sum of $\one_{K}$'s due to the assumption that $\ker\pi_{i}\neq\Lambda/\ZZ^{n}$
for every $i$.
\end{proof}

\section{\label{sec:Proof-of-First-Technical-Tool}Proof of Lemma \ref{prop:01matrixinvertible}}

Let $G$ be a finite abelian group. If $f\in L^{2}(G)$, we define
the Fourier transform $\hat{f}\in L^{2}(\widehat{G})$ by
\[
\hat{f}(\chi)=\left\langle f,\chi\right\rangle =\frac{1}{\left|G\right|}\sum_{g\in G}f\left(g\right)\overline{\chi\left(g\right)}
\]
for every $\chi\in\widehat{G}$. Since the characters of $G$ form
an orthonormal basis of $L^{2}(G)$, we have in particular that
\begin{equation}
\widehat{\psi}(\chi)=\left\langle \psi,\chi\right\rangle =\begin{cases}
1, & \psi=\chi\\
0, & \psi\neq\chi
\end{cases}\label{eq:ortho_characters}
\end{equation}
for every $\psi,\chi\in\widehat{G}$. 

Lemma \ref{prop:01matrixinvertible} is essentially a consequence
of the Poisson summation formula for finite abelian groups, stated
below. We refer the reader to \cite[Exercise 4.6]{conrad2010characters}
or \cite[Chapter 12]{terras_fourier_1999} for references. 
\begin{prop}[Poisson summation formula]
Let $G$ be a finite abelian group, let $f\in L^{2}(G)$, and let
$K$ be a subgroup of $G$. Then
\[
\frac{1}{\left|G\right|}\sum_{k\in K}f\left(k\right)=\frac{1}{\left|K^{\perp}\right|}\sum_{\chi\in K^{\perp}}\hat{f}\left(\chi\right).
\]
\end{prop}
\begin{lem}
Let $\mathcal{K}$ be a collection of subgroups of $G$ with the property
that $\{\one_{K^{\perp}}:K\in\mathcal{K}\}$ is linearly independent
in $L^{2}(\widehat{G})$. Then $\left\{ \one_{K}:K\in\mathcal{K}\right\} $
is linearly independent in $L^{2}(G)$. 
\end{lem}
\begin{proof}
Let $\mathcal{K}$ be such a collection, and suppose
\[
\sum_{K\in\mathcal{K}}\alpha_{K}\one_{K}=0
\]
for some complex numbers $\alpha_{K}$ for $K\in\mathcal{K}$. Thus
for any character $\psi\in\widehat{G}$, we obtain
\[
\sum_{K\in\mathcal{K}}\frac{\alpha_{K}}{\left|K^{\perp}\right|}\sum_{\chi\in K^{\perp}}\widehat{\psi}\left(\chi\right)=\sum_{K\in\mathcal{K}}\frac{\alpha_{K}}{\left|G\right|}\sum_{k\in K}\psi\left(k\right)=\left\langle \sum_{K\in\mathcal{K}}\alpha_{K}\one_{K},\overline{\psi}\right\rangle =0
\]
by the Poisson summation formula. On the other hand, by (\ref{eq:ortho_characters}),
the left hand side simplifies to
\[
\sum_{\substack{K\in\mathcal{K}\\
\psi\in K^{\perp}
}
}\frac{\alpha_{K}}{\left|K^{\perp}\right|}=\sum_{K\in\mathcal{K}}\frac{\alpha_{K}}{\left|K^{\perp}\right|}\one_{K^{\perp}}(\psi).
\]
It follows that the linear combination of functions
\[
\sum_{K\in\mathcal{K}}\frac{\alpha_{K}}{\left|K^{\perp}\right|}\one_{K^{\perp}}\in L^{2}(\widehat{G})
\]
is the zero function. Since the functions $\left\{ \one_{K^{\perp}}:K\in\mathcal{K}\right\} $
are assumed to be linearly independent, we get that each $\alpha_{K}=0$
which is what we wanted to show.
\end{proof}
Recall Lemma \ref{prop:01matrixinvertible}, which claims $\{\one_{K}\;:\;K^{\perp}\mathrm{\;is\;a\;cyclic\;subgroup\;of\;}\widehat{G}\}$
is linearly independent in $L^{2}(G).$
\begin{proof}[Proof of Lemma \ref{prop:01matrixinvertible}]
Let $\mathcal{K}=\{\left\langle \chi\right\rangle ^{\perp}:\chi\in\widehat{G}\}.$
By the preceding lemma, it suffices to show that set of functions
\[
\{\one_{K^{\perp}}:K\in\mathcal{K}\}=\{\one_{\left\langle \chi\right\rangle }:\chi\in\widehat{G}\}
\]
is linearly independent in $L^{2}(\widehat{G})$. 

The cyclic subgroups of $\widehat{G}$ form a partially ordered set
with respect to inclusion. Hence, by taking any linear extension of
this poset, we enumerate these subgroups as $\left\langle \chi_{1}\right\rangle ,\left\langle \chi_{2}\right\rangle ,\ldots,\left\langle \chi_{n}\right\rangle $
in such a way that $i<j$ implies there is an element of $\left\langle \chi_{j}\right\rangle $
not in $\left\langle \chi_{i}\right\rangle $. This implies that the
matrix
\[
A_{i,j}=\begin{cases}
1, & \chi_{i}\in\left\langle \chi_{j}\right\rangle \\
0, & \mbox{otherwise},
\end{cases}
\]
where $1\leq i,j\leq n$, is upper triangular with ones along the
diagonal. It follows that the functions $\one_{\left\langle \chi\right\rangle }:\chi\in\widehat{G}$
are linearly independent, as they are linearly independent when restricted
to $\left\{ \chi_{1},\ldots,\chi_{n}\right\} $.
\end{proof}

\section{\label{sec:Proof-of-Second-Technical-Tool}Proof of Theorem \ref{thm:dimensionMS}}

Let $G$ be a finite abelian group written additively, and let $H=\Hom(G,\RR/\ZZ)$.
We wish to show that the space $\Ltwoodd(H)$ of odd functions $f:H\rightarrow\CC$
is spanned by the functions $S_{g}:H\rightarrow\CC$ defined by $S_{g}(\phi)=B_{1}(\phi(g))$
for each $g\in G$. The proof of this statement is outlined in this
section, and follows the methods of \cite{reid_young_1985,morrison_terminal_1984}
by explicitly finding $\dim(L_{\mathrm{odd}}^{2}(H))$ many linearly
independent vectors in $\spann(S_{g}:g\in G)$. 

\subsection{Some preliminaries}

By the structure theorem for finitely generated abelian groups, $G\simeq H$
is isomorphic to an additive group of the form
\[
\bigoplus_{i=1}^{m}\ZZ/r_{i}
\]
where $r_{1},r_{2},\ldots,r_{m}$ are positive integers such that
$m\geq1$ and $r_{1}\mid r_{2}\mid\cdots\mid r_{m}$. Now fix a minimal
set of generators $\left\{ g_{1},\ldots,g_{m}\right\} $ of $G$,
so that every element $g\in G$ can be written uniquely as $a_{1}g_{1}+\cdots+a_{m}g_{m}$
for some integers $a_{1},\ldots,a_{m}$ satisfying $0\leq a_{i}<r_{i}$
for $i=1,2,\ldots,m$. Then the maps $\phi_{i}\in H$ for $i=1,2,\ldots,m$
defined by 
\[
\phi_{i}(g_{j})=\begin{cases}
0+\ZZ & i\neq j\\
\frac{1}{r_{i}}+\ZZ & i=j
\end{cases}
\]
are a minimal generating set for $H$ in that every $\phi\in H$ can
be written uniquely as $c_{1}\phi_{1}+\cdots+c_{m}\phi_{m}$ for some
integers $c_{1},\ldots,c_{m}$ satisfying $0\leq c_{i}<r_{i}$ for
$i=1,2,\ldots,m$. Moreover, given $g\in G$ and $\phi\in H$, if
$g=a_{1}g_{1}+\cdots+a_{m}g_{m}$ and $\phi=c_{1}\phi_{1}+\cdots+c_{m}\phi_{m}$,
then
\[
\phi\left(g\right)=\frac{a_{1}c_{1}}{r_{1}}+\cdots+\frac{a_{m}c_{m}}{r_{m}}+\ZZ.
\]
Now let $R$ denote the ring $\ZZ/r_{1}\oplus\cdots\oplus\ZZ/r_{m}$
with componentwise multiplication, so that the additive group of $R$
is isomorphic to $G$. For each $a=(a_{1},\ldots,a_{m})\in R$, define
the function $S_{a}:R\rightarrow\CC$ by
\[
S_{a}\left(c\right)=B_{1}\left(\frac{a_{1}c_{1}}{r_{1}}+\cdots+\frac{a_{m}c_{m}}{r_{m}}\right)
\]
for each $c=(c_{1},\ldots,c_{m})\in R$.

As before, let $\Ltwoodd(R)$ denote the space of functions $f:R\rightarrow\CC$
satisfying $f(-a)=-f(a)$ for all $a\in R$. Theorem \ref{thm:dimensionMS},
then, is established by proving the following proposition:
\begin{prop}
\label{prop:span_odd_functions_L2R}The functions in $\{S_{a}:a\in R\}$
span $L_{\mathrm{odd}}^{2}(R)$.
\end{prop}
Before proceeding with the proof of Proposition \ref{prop:span_odd_functions_L2R},
we review the notion of Dirichlet characters and establish the notation
to be used in the remainder of this section. A reference can be found
in \cite[Section 9.1]{montgomery_multiplicative_2006}.

\subsection{Dirichlet characters}

Let $G=\left(\ZZ/r\right)^{\times}$ for some positive integer $r$.
Then each character $\chi:G\rightarrow\CC^{\times}$ extends to a
completely multiplicative function $\chi:\ZZ\rightarrow\CC$ by setting
\[
\chi\left(n\right):=\begin{cases}
\chi\left(n+r\ZZ\right), & \gcd\left(n,r\right)=1\\
0, & \mbox{otherwise}
\end{cases}
\]
for each integer $n$. A function $\chi:\ZZ\rightarrow\CC$ is called
a \emph{Dirichlet character} if it is constructed in this manner for
some $r\geq1$ and some $\chi\in\widehat{\left(\ZZ/r\right)^{\times}}$.
The number $r$ is called the \emph{modulus} of $\chi$. We define
an equivalence relation $\sim$ on Dirichlet characters by declaring
$\chi_{1}\sim\chi_{2}$ if and only if they agree on their mutual
support. A Dirichlet character $\chi$ is called \emph{primitive}
if the support of $\chi$ contains the support of every other Dirichlet
character in the equivalence class $\left[\chi\right]$. Given a Dirichlet
character $\chi$, there exists a unique primitive Dirichlet character
in the equivalence class $\left[\chi\right]$ and it is denoted $\chi^{*}$.
A primitive character $\chi^{*}$ is said to \emph{induce} a Dirichlet
character $\psi$ if $\psi\in\left[\chi^{*}\right]$. If $\chi$ is
a Dirichlet character, then the modulus of $\chi^{*}$ is called the
\emph{conductor} of $\chi$.

\subsection{Notation}

We outline the notation used in the remainder of this section.

\subsubsection{Arithmetic functions}

Let $\NN$ denote the positive integers.
\begin{itemize}
\item $\nu_{p}:\NN\rightarrow\ZZ$ denotes the $p$-adic valuation: $\nu_{p}(k)$
is the largest exponent $\alpha$ such that $p^{\alpha}\mid k$.
\item $\divisor:\NN\rightarrow\NN$ counts the number of divisors of an
integer: we have $\divisor(k)=\prod_{p}\left(\nu_{p}(k)+1\right)$
for all $k\geq1$ where the product is over all primes $p$.
\item $\mu:\NN\rightarrow\ZZ$ is the M\"obius function.
\item $\varphi:\NN\rightarrow\NN$ is the Euler-phi function.
\item We write $\left(k,k'\right)$ for the greatest common divisor of $k$
and $k'$.
\end{itemize}
For $a=\left(a_{1},\ldots,a_{m}\right)\in\NN^{m}$, we also define
\begin{itemize}
\item $\divisor(a):=\divisor(a_{1})\divisor(a_{2})\cdots\divisor(a_{m}).$
\item $\mu(a):=\mu(a_{1})\mu(a_{2})\cdots\mu(a_{m})$.
\item $\varphi(a):=\varphi(a_{1})\varphi(a_{2})\cdots\varphi(a_{m})$.
\end{itemize}

\subsubsection{Dirichlet characters}

Let $R=\oplus_{i=1}^{m}\ZZ/r_{i}$ be the ring defined above. The
multiplicative group of units of $R$ is given by
\[
R^{\times}=\bigoplus_{i=1}^{m}\left(\ZZ/r_{i}\right)^{\times}.
\]
 Let $\widehat{R^{\times}}$ denote the group of characters of $R^{\times}$.
Each $\chi\in\widehat{R^{\times}}$ corresponds uniquely to a tuple
$\left(\chi_{1},\ldots,\chi_{m}\right)$ for which $\chi_{i}\in\widehat{\left(\ZZ/r_{i}\right)^{\times}}$
for $i=1,2,\ldots,m$, and
\[
\chi\left(a\right)=\chi_{1}\left(a_{1}\right)\cdots\chi_{m}\left(a_{m}\right)
\]
for each $a=\left(a_{1},\ldots,a_{m}\right)\in R^{\times}$. For a
character $\chi_{i}:(\ZZ/r_{i})^{\times}\rightarrow\CC^{\times}$,
we denote the corresponding Dirichlet character by $\chi_{i}:\ZZ\rightarrow\CC$.
For $\chi=\left(\chi_{1},\ldots,\chi_{m}\right)\in\widehat{R^{\times}}$,
we define
\begin{itemize}
\item $\chi^{\phantom{*}}:\ZZ^{m}\rightarrow\CC$ by $\chi^{\phantom{*}}(a_{1},\ldots,a_{m})=\chi_{1}\left(a_{1}\right)\chi_{2}\left(a_{2}\right)\cdots\chi_{m}\left(a_{m}\right).$
\item $\chi^{*}:\ZZ^{m}\rightarrow\CC$ by $\chi^{*}(a_{1},\ldots,a_{m})=\chi_{1}^{*}\left(a_{1}\right)\chi_{2}^{*}\left(a_{2}\right)\cdots\chi_{m}^{*}\left(a_{m}\right).$
\end{itemize}
Here we are denoting by $\chi_{i}^{*}:\ZZ\rightarrow\CC$ the primitive
Dirichlet character inducing $\chi_{i}:\ZZ\rightarrow\CC$.

\subsubsection{Parameters associated with $R$.}

For the ring $R$ defined above, and for each $\chi=\left(\chi_{1},\ldots\chi_{m}\right)\in\widehat{R^{\times}}$,
we define
\begin{itemize}
\item $r:=(r_{1},\ldots,r_{m})$.
\item $f_{\chi}:=\left(f_{\chi_{1}},\ldots,f_{\chi_{m}}\right)$ where $f_{\chi_{i}}$
is the conductor of $\chi_{i}:\ZZ\rightarrow\CC$. 
\item $q_{\chi}:=\left(r_{1}/f_{\chi_{1}},\ldots,r_{m}/f_{\chi_{m}}\right)$.
\end{itemize}

\subsubsection{Everything else.}

For two tuples of integers $a=\left(a_{1},\ldots,a_{m}\right),c=\left(c_{1},\ldots,c_{m}\right)\in\ZZ^{m}$,
we write $ac$ to denote the componentwise product $(a_{1}c_{1},\ldots,a_{m}c_{m})$.
If $a$ and $c$ have positive components, then we write $a\mid c$
and say $a$ \emph{divides} $c$ if $a_{i}\mid c_{i}$ for all $i=1,2,\ldots,m$.
If $a$ divides $c$, then we let $c/a:=\left(c_{1}/a_{1},\ldots,c_{m}/a_{m}\right)$.
Thus, for instance, $q_{\chi}=r/f_{\chi}$ where $r,q_{\chi},f_{\chi}$
are as above.

If $g,h:\NN^{m}\rightarrow\CC$, then let $*$ denote Dirichlet convolution
over $\NN^{m}$: 
\[
\left(g*h\right)\left(a\right):=\sum_{d\mid a}g\left(d\right)h\left(a/d\right)\quad\text{for all\;}a\in\NN^{m}.
\]

\subsection{A decomposition of $R$}

The group of units $R^{\times}$ of the ring $R$ acts on $L^{2}(R)$
as follows: for a given $f\in L^{2}(R)$, $c\in R^{\times}$, the
function $c\cdot f\in L^{2}(R)$ is defined so that
\[
(c\cdot f)(a)=f(ca)
\]
for each $a\in R$. An eigenvalue, eigenfunction pair $\left(\chi,w\right)$
of the action consists of a function $\chi:R^{\times}\rightarrow\CC$
and a nonzero function $w\in L^{2}(R)$ such that for every $c\in R^{\times}$,
\[
c\cdot w=\chi\left(c\right)w.
\]
The vector space $L^{2}(R)$ can be decomposed into a direct sum
\begin{equation}
L^{2}(R)=\bigoplus_{\chi\in\widehat{R^{\times}}}\varepsilon_{\chi}\label{eq:decomp}
\end{equation}
where, for each $\chi\in\widehat{R^{\times}}$, we denote the subspace
of eigenfunctions corresponding to $\chi$ by $\varepsilon_{\chi}$.
\begin{prop}
Let $\chi\in\widehat{R^{\times}}$. Then
\[
\varepsilon_{\chi}=\left\{ \frac{1}{\left|R^{\times}\right|}\sum_{b\in R^{\times}}\overline{\chi}(b)(b\cdot w):w\in L^{2}(R)\right\} .
\]
\end{prop}
\begin{proof}
If $w$ is an eigenfunction with eigenvalue $\chi$, then $w$ equals
the average of $\overline{\chi}(b)(b\cdot w)$ over all $b\in R^{\times}$.
Conversely, if $w\in L^{2}(R)$, then
\begin{align*}
c\cdot\sum_{b\in R^{\times}}\overline{\chi}(b)(b\cdot w)=\sum_{b\in R^{\times}}\overline{\chi}(b)(cb\cdot w) & =\chi(c)\sum_{b\in R^{\times}}\overline{\chi}(cb)(cb\cdot w)\\
 & =\chi(c)\sum_{b\in R^{\times}}\overline{\chi}(b)(b\cdot w).
\end{align*}
\end{proof}
We say that a character $\chi\in\widehat{R^{\times}}$ is \emph{even}
if $\chi\left(-1,\ldots,-1\right)=1$ and \emph{odd} if $\chi\left(-1,\ldots,-1\right)=-1$;
note that these are the only two possible values for $\chi\left(-1,\ldots,-1\right)$
since
\[
\left(\chi\left(-1,\ldots,-1\right)\right)^{2}=\chi(\left(-1,\ldots,-1\right)^{2})=\chi\left(1,\ldots,1\right)=1.
\]
Observe that the functions in $\varepsilon_{\chi}$ are odd if and
only if $\chi$ is odd. 

For a given $\chi\in\widehat{R^{\times}}$ and $a'\in R$, let
\[
w_{\chi,a'}:=\sum_{b\in R^{\times}}\chi(b)S_{a'b}\in\varepsilon_{\chi}.
\]
If $a\in\ZZ^{m}$ and $a'\in R$ is the image of $a$ under the canonical
map $\ZZ^{m}\rightarrow R$, then we also define $w_{\chi,a}:=w_{\chi,a'}$.

The next theorem, proved by Reid in \cite{reid_young_1985} for the
case when $m=1$, finds a basis of $\varepsilon_{\chi}$ in terms
of these $w_{\chi,a}$ when $\chi$ is odd.
\begin{prop}[{c.f. \cite[Theorem 5.13]{reid_young_1985}}]
\label{thm:MReigenspaces} For each odd character $\chi\in\widehat{R^{\times}}$,
there are $\divisor(q_{\chi})$ functions in $\{w_{\chi,a}:a\in\NN^{m},a\mid q_{\chi}\}$
and they are linearly independent.
\end{prop}
With this proposition, we can prove Proposition \ref{prop:span_odd_functions_L2R}
and hence Theorem \ref{thm:dimensionMS}.
\begin{proof}[Proof of Proposition \ref{prop:span_odd_functions_L2R}]
For a tuple $f=\left(f_{1},\ldots,f_{m}\right)\in\NN^{m}$, let $\hat{\varphi}_{\mathrm{odd}}\left(f\right)$
denote the number of odd characters $\chi\in\widehat{R^{\times}}$
such that $f=f_{\chi}$. Using Proposition \ref{thm:MReigenspaces}
and the decomposition (\ref{eq:decomp}), we take the union of the
sets $\left\{ w_{\chi,a}:a\mid q_{\chi}\right\} $ over all odd characters
$\chi$ to obtain $\left(\hat{\varphi}_{\mathrm{odd}}*\divisor\right)\left(r\right)$
linearly independent functions in $\Ltwoodd(R)$. We would therefore
like to show that this number is equal to $\dim\left(\Ltwoodd\left(R\right)\right)$. 

Since $*$ is associative, we get
\[
\left(\hat{\varphi}_{\mathrm{odd}}*\divisor\right)\left(r\right)=\left(\hat{\varphi}_{\mathrm{odd}}*1*1\right)\left(r\right)=\sum_{f\mid r}\left(\hat{\varphi}_{\mathrm{odd}}*1\right)\left(f\right).
\]
Now each term $\left(\hat{\varphi}_{\mathrm{odd}}*1\right)\left(f\right)$
in the sum is equal to the total number of odd characters of the group
$G_{f}:=\oplus_{i=1}^{m}\left(\ZZ/f_{i}\right)^{\times}$. This number
is equal to zero if $G_{f}$ is the trivial group, which is the case
if and only if $f_{i}$ equals one or two for every $i=1,2,\ldots,m$.
Otherwise, $\{\psi\in\widehat{G_{f}}:\psi(-1,\ldots,-1)=1\}$ is an
order two subgroup of $\widehat{G_{f}}$ and hence there are $\tfrac{1}{2}\left|\widehat{G_{f}}\right|=\tfrac{1}{2}\left|G_{f}\right|=\tfrac{1}{2}\varphi\left(f\right)$
odd characters in $\widehat{G_{f}}.$

If we let $\delta\left(f\right)=1$ whenever every component of $f$
is either 1 or 2 and zero otherwise, then we obtain
\[
\left(\hat{\varphi}_{\mathrm{odd}}*\divisor\right)\left(r\right)=\frac{1}{2}\sum_{f\mid r}\left(\varphi\left(f\right)-\delta\left(f\right)\right)=\frac{1}{2}\left(r_{1}\cdots r_{m}-2^{s}\right)
\]
where $s$ equals the number of $i\in\{1,2,\ldots,m\}$ such that
$r_{i}$ is even. Hence we obtain that the dimension of $\spann(S_{a}:a\in R)$
is at least $\tfrac{1}{2}\left(\left|R\right|-2^{s}\right)$. 

It remains to show that $\dim\left(\Ltwoodd(R)\right)=\tfrac{1}{2}\left(\left|R\right|-2^{s}\right)$.
Observe that the functions $\left\{ \one_{a}-\one_{-a}:a\in R\right\} $
span $\Ltwoodd\left(R\right)$, where $\one_{a}\in L^{2}(R)$ denotes
the indicator function of the element $a\in R$. Indeed, for any $h\in\Ltwoodd\left(R\right)$
we have
\[
h=\frac{1}{2}\sum_{a\in R}h(a)\left(\one_{a}-\one_{-a}\right).
\]
The dimension of $\spann\left(\one_{a}-\one_{-a}:a\in R\right)$ is
equal to one-half the number of elements $a\in R$ such that $a\neq-a$.
But the elements $a\in R$ for which $a=-a$ are precisely the elements
$\left(\epsilon_{1}r_{1}/2,\ldots,\epsilon_{m}r_{m}/2\right)\in R$
where each $\epsilon_{i}=0\;\text{or}\;1$ but $\epsilon_{i}=0$ for
all $i$ such that $r_{i}$ is odd. That is to say, the number of
elements $a\in R$ such that $a=-a$ is exactly $2^{s}$. We therefore
conclude
\begin{align*}
\dim\Ltwoodd\left(R\right) & =\dim\spann\left(\one_{a}-\one_{-a}:a\in R\right)\\
 & =\tfrac{1}{2}\left(\left|R\right|-2^{s}\right)\\
 & \leq\dim\spann(S_{a}:a\in R)\\
 & \leq\dim\Ltwoodd\left(R\right).
\end{align*}
and hence equality holds throughout. Since $S_{a}\in\Ltwoodd(R)$
for each $a\in R$, we conclude that $\Ltwoodd(R)=\spann(S_{a}:a\in R)$. 
\end{proof}

\subsection{Finding a basis for each eigenspace}

It therefore remains to prove Proposition \ref{thm:MReigenspaces}.
For the rest of the paper, we fix some odd $\chi\in\widehat{R^{\times}}$
and let $q:=\left(q_{1},\ldots,q_{m}\right):=q_{\chi}$ and $f:=\left(f_{1},\ldots,f_{m}\right):=f_{\chi}$. 

We start by finding an alternate representation for $w_{\chi,a}(c)$
given $a,c\in R$. This representation is based on \cite[Theorem 9.9]{montgomery_multiplicative_2006},
which expresses the generalized Bernoulli number $B_{1,\chi}$ in
terms of the Dirichlet $L$-function $L\left(s,\chi\right)$ evaluated
at $s=1$.
\begin{prop}
\label{prop:what_is_w_chi_a_c}Let $a=(a_{1},\ldots,a_{m}),c=(c_{1},\ldots,c_{m})\in R$.
Then
\[
w_{\chi,a}\left(c\right)=\frac{i}{\pi}\sum_{k\geq1}\frac{1}{k}\prod_{i=1}^{m}\left(\overline{\chi}_{i}^{*}\left(\frac{ka_{i}c_{i}}{(r_{i},ka_{i}c_{i})}\right)F_{\chi_{i}}((r_{i},ka_{i}c_{i}))\right)
\]
where $F_{\chi_{i}}(\beta)=0$ if $\beta$ does not divide $q_{i}$,
and otherwise
\begin{equation}
F_{\chi_{i}}(\beta)=\chi_{i}^{*}\left(\frac{q_{i}}{\beta}\right)\mu\left(\frac{q_{i}}{\beta}\right)\frac{\varphi(r_{i})\tau(\chi_{i}^{*})}{\varphi(r_{i}/\beta)}.\label{eq:MV}
\end{equation}
\end{prop}
The factor $\tau\left(\chi_{i}^{*}\right)$ above denotes the Gauss
sum of the primitive character $\chi_{i}^{*}$:
\[
\tau\left(\chi_{i}^{*}\right):=\sum_{t\in\left(\ZZ/f_{i}\right)^{\times}}\chi_{i}^{*}\left(t\right)e\left(t/f_{i}\right).
\]
For our purposes, the only thing we need to know about this quantity
is that it is nonzero \cite[Theorem 9.7]{montgomery_multiplicative_2006}.
\begin{proof}
Consider the quantity
\[
A:=\sum_{\substack{b\in R^{\times}\\
\theta_{acb}\notin\ZZ
}
}\chi(b)\log\left(1-e\left(\theta_{acb}\right)\right)
\]
where $e\left(x\right):=\exp\left(2\pi ix\right)$, the logarithm
is the principal branch, and 
\[
\theta_{acb}:=\frac{a_{1}c_{1}b_{1}}{r_{1}}+\cdots+\frac{a_{m}c_{m}b_{m}}{r_{m}}.
\]
In the sum, we replace $\log\left(1-e\left(\theta_{acb}\right)\right)$
with its real and imaginary parts:
\[
\log\left(1-e\left(\theta_{acb}\right)\right)=\log\left|2\sin\left(\pi\theta_{acb}\right)\right|+i\pi\left(\left\{ \theta_{acb}\right\} -1/2\right),
\]
then distribute to obtain two sums. The first of these is zero which
can be seen by noting that $\left|\sin\left(\pi\theta_{acb}\right)\right|=\left|\sin\left(\pi\theta_{-acb}\right)\right|$
and therefore we can replace each $\chi(b)$ with $\tfrac{1}{2}(\chi(b)+\chi(-b))$
which is zero since $\chi$ is odd. The second sum is therefore equal
to $A$, and from it we recover $w_{\chi,a}(c)$: 
\[
A=i\pi\sum_{\substack{b\in R^{\times}\\
\theta_{acb}\notin\ZZ
}
}\chi(b)(\{\theta_{acb}\}-1/2)=i\pi w_{\chi,a}(c).
\]
On the other hand, we use the Taylor expansion of the logarithm to
obtain
\begin{align*}
A & =\sum_{\substack{b\in R^{\times}\\
\theta_{acb}\notin\ZZ
}
}\chi(b)\sum_{k\geq1}-\frac{e(k\theta_{acb})}{k}=-\sum_{k\geq1}\frac{1}{k}\sum_{b\in R^{\times}}\chi(b)e(k\theta_{acb}).
\end{align*}
Since the double sum on the left is a finite sum of convergent series,
we may interchange the sums. The second equality holds since, after
interchanging, the terms of the inner sum for which $\theta_{acb}\in\ZZ$
sum to zero. Indeed, over such terms we may pull out $e\left(k\theta_{acb}\right)=1$
and replace each $\chi(b)$ with $\tfrac{1}{2}(\chi(b)+\chi(-b))$
which is zero as before. We may therefore write the inner sum as the
product
\[
\prod_{i=1}^{m}\left(\sum_{b_{i}\in\left(\ZZ/r_{i}\right)^{\times}}\chi_{i}\left(b_{i}\right)e\left(\frac{ka_{i}c_{i}b_{i}}{r_{i}}\right)\right).
\]
Now let $\beta_{i,k}:=(r_{i},ka_{i}c_{i})$. Applying \cite[Theorem 9.12]{montgomery_multiplicative_2006},
each factor above can be written
\[
\overline{\chi_{i}}^{*}\left(\frac{ka_{i}c_{i}}{\beta_{i,k}}\right)\chi_{i}^{*}\left(\frac{q_{i}}{\beta_{i,k}}\right)\mu\left(\frac{q_{i}}{\beta_{i,k}}\right)\frac{\varphi\left(r_{i}\right)}{\varphi\left(r_{i}/\beta_{i,k}\right)}\gauss\left(\chi_{i}^{*}\right)
\]
if $\beta_{i,k}\mid q_{i}$. Otherwise it is zero.
\end{proof}
Following Reid in \cite[Theorem 5.16]{reid_young_1985}, it is more
convenient to prove Proposition \ref{thm:MReigenspaces} by showing
that the functions
\[
v_{\chi,a}:=\sum_{d\mid a}\mu\left(d\right)\overline{\chi}^{*}\left(d\right)w_{\chi,a/d}
\]
over all $a\in\NN^{m}$ which divide $q$ are linearly independent
in $L^{2}(R)$. We can accomplish this by showing that the matrix
\[
\left(v_{\chi,a}(c)\right)_{a,c}
\]
is nonsingular, where the rows and columns of the matrix are indexed
by tuples $a,c\in\NN^{m}$ such that $a\mid q$ and $c\mid q$, and
$v_{\chi,a}(c):=v_{\chi,a}(c')$ where $c'$ is the image of $c$
under the canonical map $\ZZ^{m}\rightarrow R$. This is done over
the next three propositions. Proposition \ref{prop:Ordering} finds
an ordering of the divisors of $q$ so that: 
\begin{enumerate}
\item the indices $(a,c)$ of the antidiagonal entries of the matrix satisfy
$ac=q$.
\item The indices $(a,c)$ to the right of the antidiagonal entries satisfy
$ac\nmid q$.
\end{enumerate}
Proposition \ref{prop:v_chi_a_c_equals_0} shows that $v_{\chi,a}(c)=0$
for all $a\mid q$ and $c\mid q$ satisfying $ac\nmid q$. Finally,
this paper concludes with Proposition \ref{prop:v_chi_a_c_neq_0},
which shows that $v_{\chi,a}(c)\neq0$ for all $a,c\in\NN^{m}$ satisfying
$ac=q$ and hence the matrix is indeed nonsingular.
\begin{prop}
\label{prop:Ordering}There exists a linear ordering 
\[
a^{(1)}<a^{(2)}<\cdots<a^{(N)}
\]
of tuples in $\NN^{m}$ which divide $q$, so that:
\end{prop}
\begin{enumerate}
\item For all $i,j=1,2,\ldots,N$, $i<j$ implies $a^{(j)}\nmid a^{(i)}$. 
\item For all $i=1,2,\ldots,N$, $a^{(i)}a^{(N-i+1)}=q$.
\end{enumerate}
\begin{proof}
The tuples in $\NN^{m}$ which divide $q$ form a graded poset with
rank function given by $\mathrm{rank}\left(a_{1},\ldots,a_{m}\right)=\sum_{i=1}^{m}\sum_{p}\nu_{p}\left(a_{i}\right)$
where the inner sum is over all primes $p$. To construct our ordering,
we first specify that $a<b$ whenever $\mathrm{rank}\left(a\right)<\mathrm{rank}\left(b\right)$.
Then, we arbitrarily order the elements within each level set $\mathrm{rank}^{-1}\left(j\right)$
for each $j$ in the range $0\leq j<\mathrm{rank}\left(q\right)/2$.
If $\mathrm{rank}\left(q\right)$ is even, we further take the elements
$a$ with rank equal to $\mathrm{rank}\left(q\right)/2$ which do
not satisfy $a^{2}=q$, group them into pairs of the form $\left(a,q/a\right)$,
choose a unique representative from each such pair, and arbitrarily
order these representatives. Next, we set $q/a>q/b$ whenever $a<b$
and $\mathrm{rank}\left(a\right)=\mathrm{rank}\left(b\right)\leq\mathrm{rank}\left(q\right)/2$.
Finally, we set $a^{((N+1)/2)}=a$ if there exists $a$ which satisfies
$a^{2}=q$. The result is a linear ordering satisfying (1) and (2).
\end{proof}
\begin{prop}[{c.f. \cite[Proposition 5.17(i)]{reid_young_1985}, \cite[Lemma 4.18]{fletcher_inverting_1989}}]
\label{prop:v_chi_a_c_equals_0}If $a,c\in\NN^{m}$ divide $q$ but
$ac\nmid q$, then $v_{\chi,a}(c)=0$.
\end{prop}
\begin{proof}
Assume $a\mid q$ and $c\mid q$ but $ac\nmid q$. Then there exists
some $i\in\left\{ 1,2,\ldots,m\right\} $, $\alpha\geq0$, and prime
$p_{i}$ such that $p_{i}^{\alpha+1}\mid a_{i}c_{i}$ and $p_{i}^{\alpha}\mid q_{i}$
but $p_{i}^{\alpha+1}\nmid q_{i}$. The key insight (taken from the
above two references) is that are two different possible reasons why
$v_{\chi,a}(c)$ must equal zero, depending on whether or not $p_{i}\mid f_{i}$. 

First suppose $p_{i}\mid f_{i}$. Let $d\mid a$ and assume that $d_{i}$
is coprime to $f_{i}$. Then $p_{i}$ does not divide $d_{i}$ and
therefore $p_{i}^{\alpha+1}$ divides $a_{i}c_{i}/d_{i}$. We also
have $p_{i}^{\alpha}\mid q_{i}$ and $p_{i}\mid f_{i}$ which means
$p_{i}^{\alpha+1}\mid r_{i}$. It follows that $p_{i}^{\alpha+1}$
divides $\left(r_{i},a_{i}c_{i}/d_{i}\right)$ and hence $p_{i}^{\alpha+1}\mid\left(r_{i},ka_{i}c_{i}/d_{i}\right)$
for every $k\geq1$. Since $p_{i}^{\alpha+1}$ does not divide $q_{i}$,
it follows that $\left(r_{i},ka_{i}c_{i}/d_{i}\right)$ does not divide
$q_{i}$ for any $k\geq1$. By Proposition \ref{prop:what_is_w_chi_a_c},
then, we conclude $w_{\chi,a/d}(c)=0$ for every $d\mid a$ such that
$d_{i}$ is coprime to $f_{i}$. But the only terms in the sum
\[
v_{\chi,a}(c)=\sum_{d\mid a}\mu\left(d\right)\overline{\chi}^{*}\left(d\right)w_{\chi,a/d}(c)
\]
which can be nonzero are the ones for which $d$ is coprime to $f$
in \emph{every} component, including component $i$. This is due to
the presence of the $\overline{\chi}^{*}\left(d\right)$ term which
vanishes if this is not the case. It follows that $v_{\chi,a}(c)=0$
in the case $p_{i}\mid f_{i}$.

Now suppose $p_{i}\nmid f_{i}$. Since $p_{i}^{\alpha+1}\mid a_{i}c_{i}$,
it follows that $p_{i}$ must divide both $a_{i}$ and $c_{i}$ since
both $a_{i}$ and $c_{i}$ are divisors of $q_{i}$ and $p_{i}^{\alpha+1}\nmid q_{i}$.
In particular, $p_{i}$ must divide $a_{i}$. Now let $p:=(1,\ldots,1,p_{i},1,\ldots,1)$
and let $p'=p^{\nu_{p_{i}}(a_{i})}$ so that the $i$\textsuperscript{th}
component of $a/p'$ is not divisible by $p_{i}$. Because the presence
of the $\mu(d)$ term ensures that the sum $v_{\chi,a}(c)$ is only
over $d$ with squarefree components, we can group the sum as follows:
\begin{equation}
\sum_{d\mid\frac{a}{p'}}\left(\mu\left(d\right)\overline{\chi}^{*}\left(d\right)w_{\chi,a/d}(c)+\mu\left(pd\right)\overline{\chi}^{*}\left(pd\right)w_{\chi,a/pd}(c)\right).\label{eq:group_into_pairs_cancel_out}
\end{equation}
Since $\mu(pd)=-\mu(d)$ for every $d\mid\tfrac{a}{p'}$, it suffices
to show
\[
\overline{\chi}^{*}\left(d\right)w_{\chi,a/d}(c)=\overline{\chi}^{*}\left(pd\right)w_{\chi,a/pd}(c)
\]
for every $d\mid\tfrac{a}{p'}$ in order to establish $v_{\chi,a}(c)=0$.
By Proposition \ref{prop:what_is_w_chi_a_c}, it suffices to show
that
\begin{align*}
 & \overline{\chi}_{i}^{*}(d_{i})\overline{\chi}_{i}^{*}\left(\frac{ka_{i}c_{i}/d_{i}}{(r_{i},ka_{i}c_{i}/d_{i})}\right)F_{\chi_{i}}((r_{i},ka_{i}c_{i}/d_{i}))\\
 & =\overline{\chi}_{i}^{*}(p_{i}d_{i})\overline{\chi}_{i}^{*}\left(\frac{ka_{i}c_{i}/p_{i}d_{i}}{(r_{i},ka_{i}c_{i}/p_{i}d_{i})}\right)F_{\chi_{i}}((r_{i},ka_{i}c_{i}/p_{i}d_{i}))
\end{align*}
for every $d\mid\tfrac{a}{p'}$ and every $k\geq1$. But since $p_{i}\nmid f_{i}$
and $p_{i}\nmid d_{i}$, we have $p_{i}^{\alpha+1}\nmid r_{i}$ while
$p_{i}^{\alpha+1}\mid ka_{i}c_{i}/d_{i}$. It follows that $\left(r_{i},ka_{i}c_{i}/d_{i}\right)=\left(r_{i},ka_{i}c_{i}/p_{i}d_{i}\right)$,
and hence the above equality indeed holds for all $d\mid\tfrac{a}{p'}$
and all $k\geq1$. 
\end{proof}
\begin{prop}[{c.f. \cite[Proposition 5.17(ii)]{reid_young_1985}}]
\label{prop:v_chi_a_c_neq_0}Let $a,c\in\NN^{m}$ be divisors of
$q$ such that $ac=q$. Then $v_{\chi,a}(c)\neq0$.
\end{prop}
\begin{proof}
Suppose $d\mid a$ and each component $d_{i}$ of $d$ is squarefree
and coprime to $f_{i}$. From Proposition \ref{prop:what_is_w_chi_a_c}
we can write
\[
\mu(d)\overline{\chi}^{*}(d)w_{\chi,a/d}(c)=\frac{i}{\pi}\sum_{k\geq1}\frac{1}{k}\prod_{i=1}^{m}\left(\mu(d_{i})\overline{\chi}_{i}^{*}\left(\frac{kq_{i}}{\beta_{i,k}}\right)F_{\chi_{i}}(\beta_{i,k})\right)
\]
where $\beta_{i,k}:=(r_{i},kq_{i}/d_{i})=(r_{i},k(a_{i}/d_{i})c_{i})$.
Now consider the factor 
\begin{equation}
\mu(d_{i})\overline{\chi}_{i}^{*}\left(\frac{kq_{i}}{\beta_{i,k}}\right)F_{\chi_{i}}(\beta_{i,k}).\label{eq:the_factor}
\end{equation}
which appears in the above expression. We start by showing that, regardless
of whether or not $\beta_{i,k}$ divides $q_{i}$, expression (\ref{eq:the_factor})
simplifies to
\begin{equation}
\gauss\left(\chi_{i}^{*}\right)\cdot\frac{\varphi\left(r_{i}\right)}{\varphi\left(d_{i}f_{i}\right)}\cdot\mu\left(\left(d_{i},k\right)\right)\varphi\left(\left(d_{i},k\right)\right)\overline{\chi_{i}}^{*}\left(k\right).\label{eq:factor_simplifies_to}
\end{equation}
Observe that $\beta_{i,k}$ divides $q_{i}$ if and only if the last
equality of
\begin{equation}
\frac{q_{i}}{\beta_{i,k}}=\frac{q_{i}}{(r_{i},kq_{i}/d_{i})}=\frac{d_{i}}{(d_{i}f_{i},k)}=\frac{d_{i}}{(d_{i},k)}\label{eq:di_di_k}
\end{equation}
holds, as $d_{i}$ is coprime to $f_{i}$ by assumption. Therefore,
if $\beta_{i,k}\mid q_{i}$, then plugging in $d_{i}/(d_{i},k)$ for
$q_{i}/\beta_{k,i}$ in (\ref{eq:the_factor}) quickly yields (\ref{eq:factor_simplifies_to}).
On the other hand, if $\beta_{i,k}\nmid q_{i}$, then (\ref{eq:the_factor})
also simplifies to (\ref{eq:factor_simplifies_to}). Indeed, in this
case (\ref{eq:the_factor}) just equals zero since $F_{\chi_{i}}(\beta_{k,i})$
is zero by definition. Since $\beta_{i,k}\nmid q_{i}$, the last equation
in (\ref{eq:di_di_k}) fails to hold. This implies $k$ shares a factor
with $f_{i}$, and hence $\overline{\chi}_{i}^{*}(k)=0$. So (\ref{eq:factor_simplifies_to})
is zero as well.

We therefore can write
\[
\mu\left(d\right)\overline{\chi}^{*}\left(d\right)w_{\chi,a/d}(c)=\frac{C'_{a,\chi}}{\varphi\left(d\right)}\sum_{k\geq1}\frac{\overline{\chi}\left(k\right)g_{d}\left(k\right)}{k}
\]
where:
\begin{itemize}
\item $C'_{a,\chi}$ is a nonzero constant that depends only on $a$ and
$\chi$
\item $\overline{\chi}:\ZZ\rightarrow\CC$ is an odd Dirichlet character
defined by 
\[
\overline{\chi}\left(k\right):=\prod_{i=1}^{m}\overline{\chi_{i}}^{*}\left(k\right)
\]
(Note: we do not put a star since this Dirichlet character may not
be primitive).
\item $g_{d}:\ZZ\rightarrow\ZZ$ is the function given by
\[
g_{d}\left(k\right)=\prod_{i=1}^{m}\mu\left(\left(d_{i},k\right)\right)\varphi\left(\left(d_{i},k\right)\right).
\]
\end{itemize}
We now further simplify the right hand side above. Let $h_{d}:\ZZ\rightarrow\CC$
be the function 
\[
h_{d}\left(k\right)=\overline{\chi}\left(k\right)\left(\mu*g_{d}\right)\left(k\right)=\overline{\chi}\left(k\right)\sum_{\ell\mid k}\mu\left(\ell\right)g_{d}\left(k/\ell\right),
\]
where $*$ denotes Dirichlet convolution. For $k\geq1$, $h_{d}\left(k\right)$
is zero unless $k$ is square-free. Indeed, if $p$ is a prime such
that $p^{\alpha}$ is the highest power of $p$ dividing $k$ and
$\alpha\geq2$, then
\[
h_{d}\left(k\right)=\overline{\chi}\left(k\right)\sum_{\ell\mid\frac{k}{p^{\alpha}}}\left(\mu\left(\ell\right)g_{d}\left(\frac{k}{\ell}\right)+\mu\left(p\ell\right)g_{d}\left(\frac{k}{p\ell}\right)\right),
\]
and since $g_{d}\left(\ell\right)$ depends only on the square-free
part of $\ell$, the terms in each summand cancel each other out as
in (\ref{eq:group_into_pairs_cancel_out}). Thus we may write
\[
h_{d}\left(k\right)=\overline{\chi}\left(k\right)\sum_{\ell\mid k}\mu\left(k/\ell\right)g_{d}\left(\ell\right)=\overline{\chi}\left(k\right)\mu\left(k\right)\sum_{\ell\mid k}\mu\left(\ell\right)g_{d}\left(\ell\right).
\]
If $n_{d}(p)$ denotes the number of indices $i\in\left\{ 1,2,\ldots,m\right\} $
such that $p\mid d_{i}$, then
\[
g_{d}\left(\ell\right)=\prod_{p\mid\ell}\left(1-p\right)^{n_{d}(p)}
\]
and so
\begin{align*}
\sum_{k\geq1}\frac{h_{d}\left(k\right)}{k} & =\sum_{k\geq1}\frac{\overline{\chi}\left(k\right)\mu\left(k\right)}{k}\prod_{p\mid k}\left(1-\left(1-p\right)^{n_{d}(p)}\right)\\
 & =\prod_{p}\left(1-\frac{\overline{\chi}\left(p\right)}{p}\left(1-\left(1-p\right)^{n_{d}(p)}\right)\right)
\end{align*}
where the first product appearing above is over all primes $p$ dividing
$k$, and the second product is over all primes $p$. From the first
equality we see that the series on the left converges absolutely (and
is in fact finite) since only finitely many primes $p$ satisfy $n_{d}(p)\geq1$.
It is a basic fact of number theory \cite[Theorem 4.9]{montgomery_multiplicative_2006}
that the sum $L\left(1,\overline{\chi}\right)=\sum_{k\geq1}\overline{\chi}\left(k\right)/k$
converges and is nonzero, and since $h_{d}=\overline{\chi}(\mu*g_{d})$
we have $\overline{\chi}g_{d}=\overline{\chi}*h_{d}$ and therefore
\[
\sum_{k\geq1}\frac{\overline{\chi}\left(k\right)g_{d}\left(k\right)}{k}=\left(\sum_{k\geq1}\frac{\overline{\chi}\left(k\right)}{k}\right)\left(\sum_{k\geq1}\frac{h_{d}\left(k\right)}{k}\right).
\]
Moreover, since the components of $d$ are squarefree and $\varphi(p)=p-1$
for every prime $p$, we have
\[
\varphi\left(d\right)=\prod_{p}\left(p-1\right)^{n_{d}(p)}
\]
and therefore
\[
\mu\left(d\right)\overline{\chi}^{*}\left(d\right)w_{\chi,a/d}(c)=C_{a,\chi}\prod_{p}\gamma\left(p,n_{d}(p)\right)
\]
where $C_{a,\chi}$ is nonzero and depends only on $a$ and $\chi$
and
\[
\gamma\left(p,k\right):=\frac{1}{\left(p-1\right)^{k}}\left(1-\frac{\overline{\chi}\left(p\right)}{p}\right)+\left(-1\right)^{k}\frac{\overline{\chi}\left(p\right)}{p}.
\]

Now we find an expression for $v_{\chi,a}\left(c\right)$. We have
\[
v_{\chi,a}\left(c\right)=C_{a,\chi}\sum_{d\mid a'}\prod_{p}\gamma\left(p,n_{d}(p)\right)=C_{a,\chi}\sum_{t}N\left(t\right)\prod_{p}\gamma\left(p,t_{p}\right),
\]
where the sum on the right hand side is over all tuples of nonnegative
integers $t=\left(t_{2},t_{3},t_{5},\ldots\right)$ indexed by the
primes, $a'=(a_{1}',\ldots,a_{m}')$ where $a_{i}'$ is the largest
squarefree divisor of $a_{i}$ coprime to $f_{i}$ for $i=1,2,\ldots,m$,
and $N\left(t\right)$ counts the number of $d\mid a'$ such that
$n_{d}(p)=t_{p}$ for all primes $p$. For a given tuple $t$, we
have
\[
N\left(t\right)=\prod_{p}\binom{n_{a'}(p)}{t_{p}},
\]
thus
\[
v_{\chi,a}\left(c\right)=C_{a,\chi}\sum_{t}\prod_{p}\binom{n_{a'}(p)}{t_{p}}\gamma\left(p,t_{p}\right)=C_{a,\chi}\prod_{p}\left(\sum_{k\geq0}\binom{n_{a'}(p)}{k}\gamma\left(p,k\right)\right).
\]
For a given prime $p$, by the binomial theorem, the inner sum is
equal to 1 if $n_{a'}(p)=0$, and otherwise equal to
\[
\left(1-\frac{\overline{\chi}\left(p\right)}{p}\right)\left(1+\frac{1}{p-1}\right)^{n_{a'}(p)}.
\]
So we conclude that
\[
v_{\chi,a}\left(c\right)=C_{a,\chi}\prod_{\substack{p\\
n_{a'}(p)\geq1
}
}\left(1-\frac{\overline{\chi}\left(p\right)}{p}\right)\left(\frac{p}{p-1}\right)^{n_{a'}(p)}\neq0.
\]
\end{proof}

\section{Acknowledgments}

This project was initiated by the author at McGill University under
the supervision of Bruce Shepherd, and the author wishes to thank
him for his valuable feedback and direction. The author also thanks
Andr\'as Seb\H{o} for his comments on this work. Finally, the author
greatly appreciates the detailed suggestions and references provided
by the anonymous referees.

\bibliographystyle{plain}
\bibliography{terminal_lemma}

\end{document}